\documentclass{article}
\usepackage{graphicx} 

\usepackage[english]{babel}

\usepackage[letterpaper,top=2cm,bottom=2cm,left=3cm,right=3cm,marginparwidth=1.75cm]{geometry}

\usepackage{amsmath}
\usepackage{amsthm}
\usepackage{graphicx}
\usepackage[colorlinks=true, allcolors=blue]{hyperref}
\usepackage{amsfonts}
\usepackage{cases}
\usepackage{bm}
\usepackage{tikz}
\usepackage{subfigure}
\usepackage{float}
\usepackage[dvipsnames]{xcolor}

\newtheorem{lemma}{Lemma}[section]
\newtheorem{proposition}{Proposition}[section]
\newtheorem{definition}{Definition}
\newtheorem{theorem}{Theorem}[section]
\newtheorem{remark}{Remark}
\newtheorem{example}{Example}
\newtheorem{corollary}{Corollary}[section]

\title{The weighted Forman and Lin-Lu-Yau Ricci flow on graphs}
\author{Shuliang Bai, Shuang Liu, Xin Lai}
\date{June 2025}

\begin{document}

\maketitle
\begin{abstract}
In this paper, we propose a type of Ricci flow on graphs where the probability distribution for the Lin-Lu-Yau curvature remains constant over time, and also study the related Forman curvature flow. These two curvature flows coincide on trees. We first prove the existence and uniqueness of solutions for both curvature flows in general graphs. Then, we obtain that the normalized curvature flow on trees converges to a constant curvature metric, and under the uniform measure, a complete classification of trees can be obtained based on the convergence results.
\end{abstract}

\section{Introduction}

In differential geometry, Ricci flow is a powerful geometric evolution tool proposed by Hamilton \cite{H82}, whose most celebrated achievement is the proof of the Poincaré conjecture by Perelman \cite{P02}. On  a manifold $(M,g(t))$, Ricci flow is
\[\partial_tg_{ij} = -2 R_{ij},\]
where \( g_{ij} \) is the metric tensor and $R_{ij}$ is Ricci curvature tensor. Ricci flow achieves the homogenization of the geometry of a manifold by evolving the metric at a rate proportional to its Ricci curvature. Inspired by this, it is natural to establish corresponding Ricci flows to understand and steer the evolution of graphs. A discrete version of Ricci flow was suggested by Ollivier \cite{Ollivier}  by letting the distance on graphs evolve through coarse
Ricci curvature
\begin{equation}\label{eq:orig}
    \frac{d}{dt}d_{xy}(t)=\kappa^o_{xy}(t)d_{xy}(t),
\end{equation}
where $d_{xy}(t)$  and $\kappa^o_{xy}(t)$ are the distance and Ollivier curvature on two distinct vertices $x$ and $y$ of the graph, respectively.  The discrete Ricci flow formula serves as an analogue of the continuous Ricci flow on discrete graph structures. It drives the iterative evolution of distance through Ricci curvature, causing the network to become more geometrically structured, thereby revealing or enhancing its intrinsic community (or modular) organization.
In \cite{NLLG19}, by discretizing time, a slightly different Ricci flow 
is used for community partitioning in complex networks.
After this, lots of applications such as \cite{Farooq2017Network, Lai2022Normalized, Lai2023Deeper,Yadav2023Discrete, Zheng2026RicciGraphDTA} utilize discrete Ricci curvature to discover graph data geometry and enhance the effectiveness of graph data analysis. 

For completeness of the theory,
Bai et al. \cite{bailin} applied Lin-Lu-Yau curvature $\kappa$, as a modified Ollivier curvature, to \eqref{eq:orig} and set the weight of each edge to the distance between its two endpoints.
They
proved the existence and uniqueness of solutions to the Lin-Lu-Yau curvature flow
\begin{equation}\label{or}
   \frac{d}{dt}\omega_e(t)=-\kappa_e(t)\omega_e(t).
\end{equation}
However, when the edge weight fails to satisfy the triangle inequality under the evolution of \eqref{or}, it triggers edge removal operations during the evolution process. 
To avoid this, Ma and Yang \cite{MY24} regarded the edge weight $\omega$ as a metric and calculated the distance between the endpoints of the edges in the following manner
\begin{equation}\label{dis}
  d_e=\inf_{\gamma\in \Gamma_{e}}\sum_{h\in \gamma}\omega_h,
\end{equation}
where $\Gamma_{e}$ is a set of all paths connecting two endpoints of  $e$.
They considered the following modified equation
\begin{equation}\label{or2}
\frac{d}{dt}\omega_e(t)=-\kappa_e(t)d_e(t),
\end{equation} 
and established the existence and uniqueness of a solution to \eqref{or2}. Subsequently, numerous studies have investigated the properties of solutions to this type of equation and lots of applications have been developed. Readers are referred to references \cite{BLLL25,MY25,TMYZ25,MYZZ24} for further details.

Let $G=(V,E,m_1,m_2)$ be an undirected graph, where $m_1$ is the positive measure of the set of vertices $V$ and $m_2$ is the non-negative measure on the set of edges $E\subset V\times V$.  Lin-Lu-Yau curvature on an edge $e=(x,y)$ essentially measures the discrepancy between the Wasserstein distance between two transition probability kernels $m_x^{\epsilon},m_y^{\epsilon}$ and the distance $d$ on $G$, as follows \begin{equation*}\label{lly}
    \kappa(x,y):=\lim_{\epsilon\rightarrow 0^+}\frac{1}{\epsilon}\left(1-\frac{W(m_x^\epsilon,m_y^\epsilon)}{d(x,y)}\right).
\end{equation*}
It is worth noting that both equations \eqref{or} and \eqref{or2} involve the evolution of the transition probability kernel $m_x^\epsilon,m_y^\epsilon$ over time. In this paper, we keep the transition probability kernels fixed over time (later defined by the vertex measure $m_1$  and the edge measure $m_2$, see Section \ref{2}), and define the distance between two vertices by assigning a weight or metric $\omega$ to each edge, as shown in \eqref{dis}. The change in the weighted Lin-Lu-Yau curvature with respect to time stems only from the variation in the metric or weight. 
In this paper, similar to \eqref{or}, we consider 
\begin{equation}\label{main-equation}
    \frac{d}{dt}\omega(t,e)=-\kappa_{\omega}(t,e)\omega(t,e),
\end{equation}
where $\kappa_{\omega}$  is the weighted Lin-Lu-Yau curvature corresponding to the metric $\omega$, that is, the transition probability kernels (i.e. $m_x^\epsilon,m_y^\epsilon$) are fixed over time and the distance is defined as \eqref{dis}. Similarly to \eqref{or}, it is necessary to assume that the edge metric in the evolution of the weighted Lin-Lu-Yau curvature flow \eqref{main-equation}  has to equal precisely to the distance between its endpoints; otherwise, edge removal operations will be introduced during the evolution process. 
\begin{theorem}[see Theorem \ref{main1}] There exists a  unique positive solution to 
\eqref{main-equation}
for all $t\in(0,\infty)$ and any edge $e\in E$ with a positive initial value $\omega_0$.
\end{theorem}

\begin{remark} If one wishes to avoid edge deletion operations, the modified version of the  weighted Lin-Lu-Yau curvature flow, sharing the same form as \eqref{or2}, can be considered, for which the existence and uniqueness of the solution can also be established. The proof follows a largely similar approach of Theorem \ref{main1}. In particular, for trees, due to the cycle-free structure, the edge weight is always equal to the distance of its two endpoints.
    Therefore, the two equations \eqref{or} and \eqref{or2} are identical on trees, and can be transformed into a linear system \eqref{fm1} when using the weighted Lin-Lu-Yau  curvature $\kappa_{\omega}$. 
\end{remark}

In addition to Ollivier curvature and Lin-Lu-Yau curvature, curvature flows defined by Forman curvature and Bakry-Émery curvature have also been extensively studied, see for example \cite{CKLMPS25,HLW24,axioms5040026,10.1093/comnet/cnw030}. Other combinatorial flows have been developed to explore various aspects of discrete curvature and metric, see for example \cite{comflow, CombinatorialCalabiFlows,GLICKENSTEIN2005791}.

As well known, Ollivier curvature is an extremal problem based on optimal transport theory. From a computational perspective, cycles, and specifically triangles, exert critical impacts on Ollivier curvature. Beyond trees, lack of  cycles, it lacks a closed-form expression for general graphs, and is consequently difficult to analyze, and exhibits high computational complexity.
By comparison,  Forman curvature on graphs can be computed via a computationally trivial formula while entirely ignoring the influence of cycle structures (see \eqref{forman} in Section 2). However, neglecting cycle structures causes Forman curvature to sacrifice partial local structural information of the graph. This leads to Forman curvature being a mere substitute when other curvature computations are infeasible due to computing power. 
 Recent researches reveal that there is a strong connection between Forman curvature and Ollivier curvature, which  has been studied in \cite{JM21,BP14,TT21}. In particular, an equivalence relation between Ollivier curvature and a new form of Forman curvature for combinatorial cell complexes including general weighted graphs has been established in \cite{JM21}.
 These results demonstrate that Forman curvature possesses profound theoretical value in addition to its computational advantages. We have precisely utilized the equivalence between Forman curvature and Lin-Lu-Yau curvature to prove the existence and uniqueness of the solution to \eqref{main-equation} in Theorem \ref{main1}.

To keep the volume of the entire graph finite, one generally considers 
the normalized  weight
\[\bar{\omega}(t,e)=\frac{\omega(t,e)}{\sum_{h\in E}\omega(t,h)}.\]
This paper also explores the convergence of the solution to Lin-Lu-Yau and Forman flow equations, as well as the convergence properties of the curvature itself. The convergence of solution to the discrete-time Ollivier curvature flow  with specific surgery has been  given in \cite{RF24}. Very recently, \cite{BHLL25} has proven partial convergence results of the solution  to \eqref{or} for trees.

\begin{theorem}[see Theorem \ref{th:main}]
Under \eqref{main-equation} on trees, the weighted Lin-Lu-Yau curvature  converges to a constant as $t\rightarrow\infty$ for any $e\in E$. Moreover, the normalized metric $\bar{\omega}(t,e)$ converges to the constant curvature metric.
\end{theorem}
\begin{remark}
    The evolution of its curvature flow is fundamentally different between keeping the transition probabilities fixed versus allowing them to change over time. Compared with the convergence result in \cite{BHLL25},  there exist non-caterpillar trees whose Ricci flow \eqref{or}  does not converge to a constant curvature metric.
\end{remark}

Moreover, in the case of $m_1=m_2\equiv1$, a complete classification for the convergence of the Ricci flow on trees is established.
\begin{theorem}[see Theorem \ref{the:4.1}]
\label{coro1}
 For a tree $T=(V,E)$ with $n$ edges,  we assume that $m_1=m_2\equiv1$. Let $\omega(t,e)$ be the solution to the Ricci flow \eqref{main-equation}. Then:
 \begin{enumerate}
    \item If  $T$ is a path, then $\omega(t,e)$ for any $e\in E$  converges to $0$, and the Ricci curvature $\kappa_w(t)$ converges to a positive constant.   
    \item For $K_{1,3}$,  $w(t,e)$ for any $e\in E$  converges to a nonzero real number, and the Ricci curvature $\kappa_w(t)$ converges to $0$.
    \item If $\max_{x\in V}d(x)\geq 3$, and $T$ is not $K_{1,3}$,  then $\omega(t,e)$ for any $e\in E$ diverges to infinity, and the Ricci curvature $\kappa_w(t)$ converges to a negative constant.
 \end{enumerate}
\end{theorem}

Similarly, we can also consider the Forman curvature flow on graphs, i.e., using Forman curvature in \eqref{main-equation}. Since Forman curvature on graphs does not account for cycles, Forman curvature flow thus becomes a linear system (see \eqref{fm1}). The existence, uniqueness and convergency of solutions to Forman curvature flow on general graphs are established under the surgery in this paper, as stated in Theorem \ref{th:exist} and Theorem \ref{th:main}.

The structure of the paper is as follows: In Section 2,  the definitions of the two types of curvature and their equivalence properties are introduced. In Section 3, we prove the existence, uniqueness, and convergence of solutions to Lin-Lu-Yau curvature and Forman curvature flow equations. In Section 4, we examine and simulate the convergence of solutions of the curvature flow on several specific graphs using two special  measures. Examples demonstrate that the convergence results of the Ricci flow can vary under different measures.

\section{The weighted Ricci curvature on graphs}\label{2}
Let $G=(V,E)$ be an undirected finite graph, where $V$ is the set of vertices and $E\subset V\times V$ is the set of edges. We write $x\sim y$ when $(x,y)\in E$.  A graph is called connected if there exists a path $x=x_0\sim x_1\sim\cdots\sim x_k=y$ for any $x,y\in V$. In this paper, we always assume that $G$ is connected.
One may endow measures on a graph with $m_1:V\rightarrow \mathbb{R}_+$  the vertex measure, and $m_2:E\rightarrow \mathbb{R}_+$ the edge measure. We assume that the edge measure is symmetric, i.e. $m_2(x,y)=m_2(y,x)$ for any $(x,y)\in E$. In the existing literature, there are two typical choices of measures:
\begin{enumerate}\label{choice}
    \item \textbf{Uniform measures:} each vertex and edge has a measure \(1\), i.e., \(m_1 \equiv 1\) and \(m_2 \equiv 1\);
    \item \textbf{Normalized measures:} each vertex measure equals the total measure 
    of its incident edges, i.e., $
    m_1(x) = \sum_{y \sim x}m_2(x,y). $
\end{enumerate}
We denote
\[\operatorname{Deg}(x):=\frac{\sum_{y\sim x} m_2(x,y)}{m_1(x)}.\]
We denote by $C(V)$ the space of functions on $V$.
Let $\omega:E\rightarrow (0,\infty)$ be the metric or weight on $E$, and the path distance $d_\omega:V\times V\rightarrow [0,\infty]$ is generated by $\omega$ via
\[d_{\omega}(u,v):=\inf\left\{\sum_{k=1}^n\omega(v_k,v_{k-1}):u=v_0\sim \cdots v_n=v\right\},\]
where the infimum is taken over all paths between $u$ and $v.$ 
Let $d(x)$ be the degree at $x\in V$ denoted  by
\[d(x):=\#\{y\in V: y\sim x\}.\]
Set $C(V)$ be the space of real function on $V$, and $\ell^\infty(V)$ be the space of bounded function on $V$ with the standard norm $\|\cdot\|_\infty$.

\subsection{The weighted Lin-Lu-Yau curvature on graphs}
For sufficiently small $\epsilon > 0$, we introduce a finitely supported transition probability kernels proposed in \cite{MW}, as follows:
\begin{equation}\label{measure}
m_x^\epsilon(y):=\left\{\begin{aligned}
    &1-\epsilon \operatorname{Deg}(x),&~~ y=x,\\
    &\frac{\epsilon m_2(x,y)}{m_1(x)},&~~y\sim x,\\
    &0,&~~\mbox{otherwise.}
\end{aligned}
  \right.
\end{equation}

  \begin{remark} 
Compared to the commonly used probability measure $\mu_x^{\alpha}$ (see, for example, \cite{bailin,LLY11}),
the measure $m_x^\epsilon$ provides a more general formulation incorporating laziness. 
In particular, $m_x^\epsilon$ differs from $\mu_x^{\alpha}$ in two key respects:
\begin{itemize}
    \item The \emph{lazy coefficient} is point-dependent, given by 
    $1 - \epsilon \operatorname{Deg}(x)$, rather than a fixed constant $\alpha$;
   \item The vertex measure $m_1$ and edge measure $m_2$ are not necessarily coupled
 and may be chosen independently.
\end{itemize}
In the special case where $\operatorname{Deg}(x) \equiv 1$ and $\epsilon = 1-\alpha$, 
the two notions coincide.
\end{remark}

Suppose that \( x\neq y\in V \), \( m_x^\epsilon \) and \( m_y^\epsilon \) are two probability distributions defined as \eqref{measure}. The Wasserstein distance $W(m_x^\epsilon,m_y^\epsilon)$ is given by 
\[W(m_x^\epsilon,m_y^\epsilon):=\inf_{A \in \Pi(m_x^\epsilon, m_y^\epsilon)} \sum_{u,v \in V} d_{\omega}(u,v) A(u,v),\]
where the transport plan \( A \) transporting \( m_x^\epsilon \) to \( m_y^\epsilon \) is a mapping \( A : V \times V \to [0, 1] \) satisfying

\[
\begin{cases}
\sum_{v \in V} A(u,v) = m_x^\epsilon (u), & u \in V, \\
\sum_{u \in V} A(u,v) = m_y^\epsilon (v), & v \in V, \\
A(u,v) \geq 0.
\end{cases}
\]
And the minimum is taken over all transport plans from \( m_x^\epsilon \) to \( m_y^\epsilon \). From the standard definition of Ricci curvature, the weighted Lin-Lu-Yau Ricci curvature with respect to $\omega$ is given by, for $x\neq y$ 
\begin{equation}\label{lly}
    \kappa_{\omega}(x,y):=\lim_{\epsilon\rightarrow 0^+}\frac{1}{\epsilon}\left(1-\frac{W(m_x^\epsilon,m_y^\epsilon)}{d_\omega(x,y)}\right).
\end{equation}
One equivalent form of Lin-Lu-Yau curvature was proposed by Münch et al.\cite{MW}, which is limit-free. Denote by
      \[\nabla_{xy}f := \frac{f(x) - f(y)}{d_{\omega}(x, y)}\]
the gradient of $f$, and by
\[
\operatorname{Lip}(K) := \{ f \in C(V) : \|\nabla f\|_\infty \leq K \}
\]
the set of all $K$-Lipschitz functions ($K > 0$) on $V$ with respect to the weighted
graph metric $d_\omega$. Then, the weighted Lin-Lu-Yau curvature can be expressed as
\begin{equation*}
\kappa_{\omega}(x,y)=\inf_{f\in \mathcal{F}_{\omega} }\nabla_{xy}\Delta f,
\end{equation*}
where $\mathcal{F}_{\omega}:=\{f\in \operatorname{Lip}(1),\nabla_{xy}f=1\}$, and the Laplacian is denoted by 
\[\Delta f(x)=\frac{1}{m_1(x)}\sum_{y\sim x}m_2(x,y)(f(y)-f(x)),~~\forall x\in V.\]

\subsection{The weighted Forman curvature for  cell complex}\label{weightform}
We  recall the definition of a combinatorial cell complex, which was proposed in \cite{F03,JM21}. Let $X=\cup_{k\geq 0}X_k$ ($X_k$ denotes the set of $k$-dimensional cells) be a finite set and $C(X)$ be the function space $\mathbb{R}^X$. 
Let $m:X\rightarrow (0,\infty)$ be a positive function, interpreted as a measure on the cells. A  cell complex is the tuple $K=(X,\delta,m)$, where $\delta:C(X_k)\rightarrow C(X_{k+1})$ is the co-boundary operator. By abuse of notation, $x = \mathbf{1}_x \in C(X)$ is written for $x \in X$.  For all $v, z \in X$ and all $k \in N_0$, the action of the co-boundary operator satisfies
\[\delta v(z)=\{-1,0,1\}.\]
Moreover, $\delta v(z)=0$ if $\dim(z)-\dim(v)\neq1$, and $\delta v(z)\neq0$ iff $z\succ v$, which denotes that \( x \) is a  face of \( z \).
Moreover,  a cell complex needs to satisfy the following compatibility conditions. For all $v, z, z' \in X,$
\begin{itemize}
    \item $|\{w\in X_0:\delta w(x)=1\}|=|\{w\in X_0: \delta w(x)=-1\}|=1$ for all $x\in X_1$.
    \item If $\dim(z)-\dim(v)=2$ and $\{x:v\prec x\prec z\}\neq \emptyset$, then
    \[|\{x:v\prec x\prec z\}|=2\]
    and 
    \[\{-1,1\}\subset\{\delta v(x)\delta x(z):x\in X\}. \]
    \item For all $x,y\prec z, $ there is a sequence $(x=x_0,\cdots,x_n=y)$ with $\delta v_k(v_{k-1})+\delta v_{k-1}(v_k)\neq 0$ for all $k=1,\cdots,n$.
    \item If $\{x:x\prec z\}=\{x:x\prec z'\}$ and $\dim(z)\geq 1$, then $z=z'.$
\end{itemize}
$X$ is equipped with the scalar product via the measure $m$ 
\[\langle f,g\rangle:=\sum_{x\in X}f(x)g(x)m(x), \forall f,g\in C(X).\]
Based on this, the adjoint $\delta^*:C(X_{k+1})\rightarrow C(X_k)$ is defined by
\[\delta^*z(x)=\frac{1}{m(x)}\langle\delta^*z,x\rangle=\frac{1}{m(x)}\langle z,\delta x\rangle=\frac{m(z)}{m(x)}\delta x(z), \forall x,z\in X.\]
By linear extension,
\[\delta^*f(x)=\sum_zf(z)\delta^*z(x).\]
The  Hodge Laplacian $H : C(X_k) \rightarrow C(X_k)$ is defined by
\[H=\delta \delta^*+\delta^*\delta.\]

For every \( x \in X_1 \), there is a unique pair \( (u, v) \) with \( \delta u(x) = -1 \) and \( \delta v(x) = 1 \). The notation \( x = (u, v) = (v, u) \) is adopted, and then \( m(u,v) = m(v,u) = m(x) \) is imposed, with \( m(u,v) = 0 \) if there is no \( x \succ u, v \) or \( u = v \). It turns out that $1$-dimensional cell complexes $(X_0\cup X_1,\delta, m)$ and simple weighted graphs $(V,E,m_1,m_2)$ coincide. Hereafter, we denote $m_i$ by the measure of $X_i$.  Moreover, \( H \) on \( C(X_0) \) is the usual graph Laplacian, for any $u\in X_0$,
\[H f(u)=\frac{1}{m_1(u)}\sum_{v\sim v}m_2(u,v)(f(v)-f(u)), f\in C(X_0).\]
Further details can be found in \cite{JM21}.

Diverging from Forman's original definition based on decomposing the Hodge Laplacian $H$ as a sum of a minimally diagonally dominant operator serving as Bochner Laplacian and a diagonal operator \cite{F03}, J. Jost et. al \cite{JM21} redefine it through the diagonally dominant part of the Hodge Laplacian $H$,  thus naturally extending to the weighted case as follows.
\begin{definition}
For any $x\in X$, the weighted Forman curvature is
\[F_\omega(x)=Hx(x)-\sum_{y\neq x}\frac{\omega(y)}{\omega(x)}|Hy(x)|.\]
\end{definition}
For any $e\in X_1$, the weighted Forman curvature for $e$ is
\begin{equation}\label{Forman0}
F_{\omega}(e)=\sum_{u\prec e}\frac{m_2(e)}{m_1(u)}+\sum_{f\succ e}\frac{m_3(f)}{m_2(e)}-\sum_{e'\neq e}\frac{\omega(e')}{\omega(e)}\left|\sum_{u\prec e,e'}\frac{m_2(e')}{m_1(u)}-\sum_{f\succ e,e'}\frac{m_3(f)}{m_2(e)}\right|.
\end{equation}
If we consider a 1-dimensional cell complex, which corresponds to a graph. In particular, $X_0$ denotes the set of vertices $V$, $X_1$ denotes the set of edges $E$, and there are no 2-dimensional cell complexes. Let $e=(u,v)$, we denote $e_u\sim u$ iff $e_u\succ u$. The weighted Forman curvature on graphs reads as
\begin{equation}\label{forman weight}
  F_{\omega}(e)=\frac{m_2(e)}{m_1(u)}+\frac{m_2(e)}{m_1(v)}-\sum_{e_u\sim u,e_u\neq e}\frac{m_2(e_u)}{m_1(u)}\frac{\omega(e_u)}{\omega(e)}-\sum_{e_v\sim v,e_v\neq e}\frac{m_2(e_v)}{m_1(v)}\frac{\omega(e_v)}{\omega(e)}.
\end{equation}
If the set $\{e_u\in E: e_u\sim u,e_u\neq e\}$ is empty, then the corresponding summation is zero.
With the choice
\[m_i=\frac{1}{w_i}, i=1,2, ~\mbox{and}~ \omega=\sqrt{w_2},\]
\eqref{forman weight} can be expressed in the form
\begin{equation}\label{forman}
F_{w}(e)=\frac{w_1(u)}{w_2(e)}+\frac{w_1(v)}{w_2(e)}-\sum_{e_u\sim u,e_u\neq e}\frac{w_1(u)}{\sqrt{w_2(e_u)w_2(e)}}-\sum_{e_v\sim v,e_v\neq e}\frac{w_1(v)}{\sqrt{w_2(e_v)w_2(e)}},
\end{equation}
which equals the original Forman curvature (see \cite{F03}) divided by $w_2(e).$ Let $w_1=w_2\equiv 1$ in \eqref{forman}, this reduces to the unweighted Forman curvature
\begin{equation*}\label{unforman}
F(e)=4-d(u)-d(v),
\end{equation*}
where $d(u)$ is the number of neighborhoods of $u$.

\subsection{The relationship between Lin-Lu-Yau curvature and Forman curvature}
Lin-Lu-Yau curvature $\kappa_\omega$ of $x\in X$ with respect to $\omega$ can be rewritten as 
\begin{equation*}\label{ollivier-cell}
\kappa_\omega(x)=\frac{1}{\omega(x)}\inf_{\delta f(x)=\omega(x),|\delta f|\leq \omega}\delta \delta^*\delta f(x).
\end{equation*}
If $e \in X_1$, then $\kappa_\omega(e)$ is the usual Lin-Lu-Yau curvature of the edge $e$ with respect to $\omega$, see \eqref{lly}.

Let $(X,\delta,m)$ be a  1-dimensional cell complex. A {\it cycle} is an injective path $(v_0, \cdots, v_{n-1})$ of vertices $v_i\in X_0$ with $v_0 \sim v_{n-1}$ with $n\geq3$. Two cycles $(v_i)^{n-1}_{i=0}$ and $(w_i)^{n-1}_{i=0}$ are identified if $w_i = v_{k\pm i ~\mbox{\small{mod}}~ n}$ for some $k\in N$. 
Let $Y_2$ be the set of cycles and the non-negative function $m$ be the measure of $Y_2$. Set $X_2:=\{y\in Y_2: m(y)>0\}$ be the set of 2-cells.
For $x=(u,v)\in X_1$ and $v=(v_0, \cdots, v_{n-1})\in X_2,$ set
\[\delta x(v) =\left\{\begin{aligned}
   1,~~ &v=v_k,w=v_{k+1}\\
   -1,~~ &v=v_k,w=v_{k-1}\\
   0,~~ &\mbox{otherwise}.
\end{aligned}
  \right.\]
Then, $(X\cup X_2, \delta, m)$ is a 2-dimensional cell complex, and also a regular weighted CW complex. Through extending the 1-skeleton, optimizing the measures of 2-cells enables Lin-Lu-Yau and Forman curvatures to achieve consistency, as follows.
\begin{theorem}[see \cite{JM21}]\label{th:equi}
Let $G=(X,\delta,m)$ be a 1-dimensional cell complex and $\omega:X\rightarrow(0,\infty)$ non-degenerate. Then for any $e\in X_1$,
\begin{equation*}\label{forman-ollivier}
    \kappa_{\omega}(e)=\max_KF_{K,\omega}(e),
\end{equation*}
where the maximum is taken over all 2-dimensional cell complexes $K$ having $G$ as the 1-skeleton.
\end{theorem}

\begin{remark} \label{tree}
Let $G=(V,E,m_1,m_2)$ be a tree. Since there are no cycles in a tree, Lin-Lu-Yau curvature and Forman curvature coincide. That is, for any $e\in E$,
\begin{equation*}\label{kappa}
    \kappa_\omega (e) = F_\omega (e).
\end{equation*}
Indeed, the weighted Lin-Lu-Yau curvature for $e=(u,v)$ on a tree can be computed by \eqref{lly} as
\begin{equation}\label{1-dim}
  \kappa_{\omega}(e)=\frac{m_2(e)}{m_1(u)}+\frac{m_2(e)}{m_1(v)}-\sum_{e_u\sim u,e_u\neq e}\frac{m_2(e_u)}{m_1(u)}\frac{\omega(e_u)}{\omega(e)}-\sum_{e_v\sim v,e_v\neq e}\frac{m_2(e_v)}{m_1(v)}\frac{\omega(e_v)}{\omega(e)}.  
\end{equation}
If the set $\{e_u\in E: e_u\sim u,e_u\neq e\}$ is empty, then the corresponding summation is zero. This form exactly coincides with the weighted Forman curvature \eqref{forman weight} on general graphs.
\end{remark}


\section{The weighted Ricci flow on graphs}
In this section, we investigate the following weighted Ricci flow, for any $e\in E$,
\begin{equation}\label{flow-equation}
\left\{\begin{aligned}
    \frac{d}{dt}\omega(t,e)&=-R_\omega(t,e)\omega(t,e),~~ t>0,\\
    \omega(0,e)&=\omega_0(e)
\end{aligned}
  \right.
\end{equation}
with the initial value $\omega_0(e)>0$ for any $e\in E$, where the weighted Ricci curvature $R_\omega$ can be the weighted Lin-Lu-Yau curvature $\kappa_{\omega}$ (see \eqref{lly}), and the weighted Forman curvature without faces $F_{\omega}$(see \eqref{forman weight}). Define the normalized weight
\[\bar{\omega}(t,e)=\frac{\omega(t,e)}{\sum_{e\in E}\omega(t,e)}.\]
Observing that uniform scaling (amplification or reduction) of the weight does not alter these two types of curvature. Similar to \cite{bailin}, the normalized weight satisfies, for any $e\in E$,
\begin{equation}\label{flow-equation-nor}
\left\{\begin{aligned}
    \frac{d}{dt}\bar{\omega}(t,e)&=-R_\omega(t,e)\bar{\omega}(t,e)+\bar{\omega}(t,e)\sum_{h\in E}R_\omega(t,h)\bar{\omega}(t,h),~~ t>0,\\
    \bar{\omega}(0,e)&=\bar{\omega}_0(e).
\end{aligned}
  \right.
\end{equation}
In the evolution of the curvature flows, we assume that $\omega$ remains non-degenerate over time,  i.e., every edge $e = (x,y)\in E$ is the unique shortest path between its vertices; otherwise, it will trigger the edge deletion operation. That is, at time $t$, edges satisfying 
\[\omega_e\geq d_\omega(x,y)\]
are removed before proceeding with the subsequent evolution.

\subsection{The weighted Lin-Lu-Yau Ricci flow}
In this section, we investigate the existence and uniqueness of the solution to the weighted Lin-Lu-Yau flow by setting $R_\omega$ in \eqref{flow-equation} to be $\kappa_{\omega}$ (see \eqref{lly}). We present a new proof of the existence and uniqueness of the solution to the weighted Lin-Lu-Yau flow for graphs by examining the relationship between Lin-Lu-Yau curvature and Forman curvature (see \eqref{th:equi}), as well as utilizing the expression of the Forman curvature \eqref{Forman0}.

\begin{theorem}\label{main1}
Let $G=(V,E,m_1,m_2)$ be a graph. 
For the weighted Lin-Lu-Yau flow
\begin{equation}\label{ollivier}
    \frac{d}{dt}\omega(t,e)=-\kappa_\omega(t,e)\omega(t,e),
\end{equation}
there exists a unique positive solution for any  $t>0$ and $e \in E$ with a positive initial value $\omega_0$.
\end{theorem}

\begin{proof} We will not distinguish between the measures on cells in notation and will uniformly use $m$.
Let $G=(X_0\cup X_1,\delta,m)$ be a 1-dimensional cell complex. Set $X_2:=\{f ~\mbox{is a cycle}: m(f)>0\}$ be the set of 2-cells. For a fixed  $Y\subset X_2$, let $K=(X_0\cup X_1\cup Y,\delta,m)$ be a 2-dimensional cell complex with $G$ as its 1-skeleton. The edges in the graph can be divided into three categories based on their relationship with a given edge \(e\in X_1\): those that belong to a common cycle with \(e\) (denoted by \(\mathrm{I}_e\)), those that share a vertex with \(e\) but do not lie on a common cycle (denoted by \(\mathrm{II}_e\)), and those that are disjoint from \(e\) (which do not appear in the computation in $F_{K,\omega}(e)$). Notice that for any \(e'\in\mathrm{II}_e\),
\[\left|\sum_{u\prec e,e'}\frac{m(e')}{m(u)}-\sum_{f\succ e,e'}\frac{m(f)}{m(e)}\right|=\sum_{u\prec e,e'}\frac{m(e')}{m(u)}.\]
And, for any \(e'\in\mathrm{I}_e\),  we have  when $m(f)$ for any $f\in Y$ is sufficiently large, 
\[\left|\sum_{u\prec e,e'}\frac{m(e')}{m(u)}-\sum_{f\succ e,e'}\frac{m(f)}{m(e)}\right|=\sum_{f\succ e,e'}\frac{m(f)}{m(e)}-\sum_{u\prec e,e'}\frac{m(e')}{m(u)},\]
it follows that for any $e=(u,v)$, 
\begin{equation*}
\begin{split}F_{K,\omega}(e)
&=\sum_{u\prec e}\frac{m(e)}{m(u)}-\sum_{e'\in \mathrm{II}_e}\frac{\omega(e')}{\omega(e)}\sum_{u\prec e,e'}\frac{m(e')}{m(u)}+\frac{\sum_{f\succ e}m(f)}{m(e)}-\sum_{e'\in\mathrm{I}_e}\frac{\omega(e')}{\omega(e)}\left(\sum_{f\succ e,e'}\frac{m(f)}{m(e)}-\sum_{u\prec e,e'}\frac{m(e')}{m(u)}\right)\\
&=A+\frac{1}{m(e)\omega(e)}\left(\sum_{f\succ e}\omega(e)m(f)-\sum_{e'\in\mathrm{I}_e}\omega(e')\sum_{f\succ e,e'}m(f)\right)\\
&=A+\frac{1}{m(e)\omega(e)}\sum_{i=1}^k\left(\omega(e)-\sum_{e'\in\mathrm{I}_e, f_i\succ e,e'}w(e')\right)m(f_i),
\end{split}
\end{equation*}
where $k$ is the number of 2-cells containing $e$, and  \[A:=\sum_{u\prec e}\frac{m(e)}{m(u)}+\sum_{e'\in \mathrm{I}_e}\frac{\omega(e')}{\omega(e)}\sum_{u\prec e,e'}\frac{m(e')}{m(u)}-\sum_{e'\in \mathrm{II}_e}\frac{\omega(e')}{\omega(e)}\sum_{u\prec e,e'}\frac{m(e')}{m(u)}\]
is independent of $f_i$ for any $i$. Combining the following triangle inequality under the surgery, for any $f_i$,
\[\omega(e)\leq \sum_{e'\in\mathrm{I}_e, f_i\succ e,e'}w(e'),\]
we have $F_{K,\omega}(e)$ is non-increasing with respect to $m(f_i)$ when 
\[m(f_i)>\max_{\{e'_i\in\mathrm{I}_e:e_i'\sim e;e_i',e\prec f_i\}}\frac{m(e_i')m(e)}{\min\{m(u),m(v)\}}.\]
Let $M=\max_{e'\in X_1}\frac{m(e')m(e)}{\min\{m(u),m(v)\}}$ and $K_M$ be $K$ with $m(f)\leq M$ for any $f\in Y$. We have 
\[\max_KF_{K,\omega}(e)=\max_{Y\subset X_2}\max_{f\in Y}F_{K,\omega}(e)=\max_{Y\subset X_2}\max_{f\in Y, m(f)\leq M}F_{K,\omega}(e)=\max_{K_M}F_{K_M,\omega}(e).\]

For a fixed $K_M$, by \eqref{Forman0}, we can simplify $F_{K_M,\omega}(e)$ as 

    \[ F_{K_M,w}(e) = C_1(e) - \sum_{e' \neq e} C_2(e')\frac{\omega(e')}{\omega(e)}, \]
where \( C_1(e) = \sum_{u\prec e}\frac{m(e)}{m(u)}+\sum_{f\succ e}\frac{m(f)}{m(e)} \) and \( C_2(e') = \left|\sum_{u\prec e,e'}\frac{m(e')}{m(u)}-\sum_{f\succ e,e'}\frac{m(f)}{m(e)}\right|\) are both non-negative constants independent of \( \omega \). Let $\min_{u\in V}m(u)=c_1$, $\max_{e\in E}m(e)=c_2$ and $\min_{e\in E}m(e)=c_3$. 
Notice that $C_1(e),C_2(e)$ have a uniformly upper bound 
\[C:=\frac{c_2}{c_1}|X_0|+\frac{M}{c_3}|X_2|,\]
where $|X_i|$ represents the number of $i$-dimensional cells on $K$. 

Let $T>0$. First, we {\it claim} that  for any $t>0$, $\omega_e(t)$ is  bounded on $[0,T]$, and its lower bound is greater than zero, as follows:
\begin{equation}\label{bounds}
    \omega_e(0)e^{-CT}\leq \omega_e(t)  \leq  \left(\sum_{h \in E} \omega_h(0)\right)e^{C|X_1|T}
\end{equation} 
To prove this, we have for all \( e \in E \),
\[
-\frac{C\sum_{e'\in \mathrm{I}_e\cup \mathrm{II}_e}\omega_{e'}(t)}{\omega_e(t)} \leq \kappa_e(t) \leq C
\]
by the bounds of $F_{K_M,\omega}(e)$ and the equivalence between these two curvatures.
Thus 
\[
-C\omega_e(t)\leq -\kappa_e(t)\omega_e(t) \leq C\sum_{e\in E}\omega_{e}(t) . 
\]
On one hand, integrating the following inequality

\[
\omega'_e(t) = -\kappa_e(t)\omega_e(t) \geq -C\omega_e(t),
\] 
to get 
\( \omega_e(t) \geq \omega_e(0)e^{-CT} \)
on \(  [0, T]\). On the other hand, for any $t \in [0, T]$,
\[
\partial_t\sum_{e \in E} \omega_e(t) \leq C|X_1|\sum_{e \in E} \omega_e(t),
\]
which implies
\[
\omega_e(t) \leq \sum_{h \in E} \omega_h(t) \leq \left(\sum_{h \in E} \omega_h(0)\right)e^{C|X_1|T}.
\]
This completes the proof of the claim.

Next, we prove that $\kappa_{\omega}(e)$ is a $L$-Lipschitz function with respect to $\omega$ on $[0,T]$. From \eqref{bounds}, we may assume that there exists $\delta>1$ such that  for any $e\in X_1$ and all $i=1,2$,
\begin{equation*}\label{con}
    \delta^{-1}\leq w_i(e)\leq \delta.
\end{equation*}
And let 
\[\|\bm \omega_1-\bm \omega_2\|_{\infty}=\sup_{e\in X_1}|\omega_1(e)-\omega_2(e)|.\]
Then, 
\[\begin{split}
    |F_{K_M,\omega_1}(e)-F_{K_M,\omega_2}(e)|
    &\leq \sum_{e' \neq e}C_2(e') \left|\frac{\omega_1(e')}{\omega_1(e)}-\frac{\omega_2(e')}{\omega_2(e)}\right|\\
    &=\sum_{e' \neq e}C_2(e') \frac{\left|\omega_1(e')\omega_2(e)-\omega_2(e')\omega_1(e)\right|}{\omega_1(e)\omega_2(e)}\\
    &\leq\sum_{e' \neq e}C_2(e') \frac{\omega_1(e')|\omega_2(e)-\omega_1(e)|+\omega_1(e)|\omega_1(e')-\omega_2(e')|}{\omega_1(e)\omega_2(e)}\\
    &\leq L\|\bm \omega_1-\bm \omega_2\|_{\infty}
\end{split}\]
with $L=2\delta^3(|X_1|-1)C$, which completes the claim. From it, $\kappa_{\omega}(e)$ is a $L$-Lipschitz function with respect to $\omega$ by 
\[|\kappa_{\omega_1}(e)-\kappa_{\omega_2}(e)|=\left|\max_{K_M}F_{K_M,\omega_1}(e)-\max_{K_M}F_{K_M,\omega_2}(e)\right|\leq \max_{K_M}\left|F_{K_M,\omega_1}(e)-F_{K_M,\omega_2}(e)\right|\leq L\|\bm \omega_1-\bm \omega_2\|_{\infty}.\]
Moreover,
\[|\kappa_{\omega}(e)|=|\max_{K}F_{K,\omega}(e)|\leq C+\delta^2C.\]
By Theorem \ref{th:equi}, we obtain
\[|\kappa_{\omega_1}(e)\omega_1(e)-\kappa_{\omega_2}(e)\omega_2(e)|\leq |\kappa_{\omega_1}(e)||\omega_1-\omega_2|+\omega_2|\kappa_{\omega_1}(e)-\kappa_{\omega_2}(e)|\leq (C(1+\delta^2)+L)\|\bm \omega_1-\bm \omega_2\|_{\infty}.\]
Thus $\kappa_{\omega}\omega$ is locally Lipschitz  with respect to $\omega$ on $[0,T]$. According to the Picard-Lindelöf Theorem, there exists a unique solution to \eqref{ollivier} within $[0,T]$.

At last, we prove the long time existence of $\omega$. Define
\[
T^* = \sup\{T : \text{Equation \eqref{ollivier} has a unique solution on } [0, T]\}.
\]
Denote \( \phi(t) = \min\{\omega_{e_1}(t), \omega_{e_2}(t), \cdots, \omega_{e_n}(t)\} \) and $\Phi(t) = \max\{\omega_{e_1}(t), \omega_{e_2}(t), \cdots, \omega_{e_n}(t)\}$. Suppose that \( T^* < +\infty \),  we have either 
\begin{equation}\label{mini}
    \liminf_{t \to T^*} \phi(t) = 0
\end{equation}
or
\begin{equation}\label{max}
    \limsup_{t \to T^*} \Phi(t) = +\infty. 
\end{equation}
By the bounds \eqref{bounds} of $\omega$ on $[0,T]$, we obtain 
\[ \phi(t) \geq \phi(0)e^{-CT^*} \] 
on \( t \in [0, T^*)\), contradicting \eqref{mini}. 
And
\[
\Phi(t) \leq \sum_{e \in E} \omega_e(t) \leq \left(\sum_{e \in E} \omega_e(0)\right)e^{CnT^*},
\]
which contradicts \eqref{max}. Therefore \( T^* = +\infty \), which completes the proof.
\end{proof}

\subsection{The weighted Forman Ricci flow on graphs}
In this section, We investigate the existence, uniqueness, and convergence of solutions to the weighted Forman-Ricci flow \eqref{flow-equation} on $G$ by setting $R_\omega$ to be the weighted Forman curvature $F_{\omega}$(see \eqref{forman weight}), which gets, for any $e\in E, t>0$
\begin{equation}\label{fm1}
  \frac{d}{dt}\omega(t,e)=-\left(\frac{m_2(e)}{m_1(u)}+\frac{m_2(e)}{m_1(v)}\right)\omega(t,e)+\sum_{e_u\sim u,e_u\neq e}\frac{m_2(e_u)}{m_1(u)}\omega(t,e_u)+\sum_{e_v\sim v,e_v\neq e}\frac{m_2(e_v)}{m_1(v)}\omega(t,e_v).  
\end{equation}
By Remark \ref{tree}, all discussions and conclusions in this section apply to Lin-Lu-Yau Ricci flow \eqref{ollivier} on a tree.
 

 For notational simplicity, we introduce a matrix formulation of the Forman Ricci flow  in subsequent analysis. Let $E=\{e_1,e_2,\cdots,e_{n}\}$ and 
\[\mathbb{R}_+^{n}=\{\bm{\omega}=(\omega_1,\omega_2,\cdots,\omega_{n}),\omega_i>0, i=1,2,\cdots, n\}.\]
The weight vector is $\bm{\omega}:(0,\infty)\times E\rightarrow \mathbb{R}_+^{n}.$ The matrix form of the weighted Forman-Ricci flow \eqref{fm1} reads as
\begin{equation}\label{flow0}
\left\{\begin{aligned}
   \frac{d}{dt} \bm{\omega}&=\textbf{F}\bm{\omega},~~ t>0\\
    \bm{\omega}(0)&=\bm{\omega}_0.
\end{aligned}
  \right.
  \end{equation}
Let $e_i=(u_i,v_i)$, the matrix \textbf{F} satisfies, for any $i$,
\[\textbf{F}_{ii}=-\left(\frac{m_2(e_i)}{m_1(u_i)}+\frac{m_2(e_i)}{m_1(v_i)}\right),\]
and for $i\neq j$,
\[\textbf{F}_{ij}=\left\{\begin{aligned}
   &\frac{m_2(e_j)}{m_1(u_i)},& e_j\sim u_i\sim e_i,\\
   &\frac{m_2(e_j)}{m_1(v_i)},& e_j\sim v_i\sim e_i,\\
   &0,& \mbox{otherwise}.
\end{aligned}
  \right.\] 

\begin{lemma}[Existence and Uniqueness for Linear Time-Invariant Systems]\label{th:ex}
Consider the linear time-invariant system:
\begin{equation*}\label{linear}
\dot{\bm x}(t) = \textbf{A} \bm x(t), \quad \bm x(0) = \bm x_0,
\end{equation*}
where \( \textbf{A} \in \mathbb{R}^{n\times n}\) is a constant matrix. For any given initial value vector \( \bm x_0 \in \mathbb{R}^n\), there exists a unique solution to this system defined for all time \( t \in (-\infty, \infty) \).
   Moreover, this unique solution is given by the matrix exponential:
    \[
      \bm x(t) = e^{t\bm A} \bm x_0:= \sum_{k=0}^{\infty} \frac{(t\bm A)^k}{k!}\bm x_0.  
    \]
\end{lemma}

\begin{theorem}\label{th:exist}
The weighted Forman Ricci flow \eqref{flow0} on general graphs has a unique positive solution \(\bm \omega(t) = e^{t\bm F} \bm \omega_0\) for all $t\in(0,\infty)$ with the positive initial value $\bm w_0>\bm 0$.
\end{theorem}
\begin{proof}
From Lemma \ref{th:ex}, the weighted Forman Ricci flow \eqref{flow0} has a unique  solution for all $t\in(0,\infty)$ with
\(\bm \omega(t) = e^{t\textbf{F}} \bm \omega_0,\)
for a given $\bm w_0>\bm 0$.

Next, we investigate the positivity of $\bm \omega(t)$ for any $t>0$. Since the off-diagonal elements of \( \textbf{F}\)  are non-negative. We can choose  \(\alpha >\max_i |\textbf{F}_{ii}|\) to construct the non-negative matrix 
\[ \textbf{B} = \textbf{F} + \alpha \bm I \]
with \( \textbf{B}_{ii} = \textbf{F}_{ii} + \alpha> 0 \), and 
the off-diagonal elements of \( \textbf{B} \) are \( \textbf{F}_{ij} \geq 0 \) for any $i\neq j$. Therefore,
    \[
    e^{t\textbf{F}} = e^{t(\textbf{B} - \alpha \textbf{I})} 
           = e^{-\alpha t} e^{t \textbf{B}}.
    \]
As each term \( \frac{(t\textbf{B})^k}{k!} \) in the series is non-negative and $\left( \frac{(t\textbf{B})^k}{k!} \right)_{ii}>0$ for any $i$, then $e^{t\textbf{B}}\geq \bm 0$ and $(e^{t\textbf{B}})_{ii}>0$ for any $i$. Therefore, with $\bm w_0>\bm 0$, 
\[\bm \omega(t) = e^{t\textbf{F}} \bm \omega_0= e^{-\alpha t} e^{t \textbf{B}}\bm \omega_0>\bm 0.\]
\end{proof}


Observing that the coefficient matrix $\mathbf{F}$ is symmetric if and only if $m_2(e)$ is a constant with respect to $e\in E$.
To overcome the asymmetry of  $\mathbf{F}$ in the system \eqref{flow0} for a general measure $m_2$, we let 
\begin{equation}\label{relation}
\bm{\tilde{\omega}}(t)=\mathbf{M} \bm\omega(t),\quad t\in[0,\infty),
\end{equation}
and $\mathbf{M}:=\operatorname{diag}(\sqrt{m_2(e_1)},\sqrt{m_2(e_2)},\cdots,\sqrt{m_2(e_n)})$.
Rewrite the Forman curvature flow \eqref{flow0} to get 
\begin{equation}\label{flow1}
\left\{\begin{aligned}
   \frac{d}{dt} \bm{\tilde{\omega}}(t)&=\textbf{\~{F}}\bm{\tilde{\omega}}(t),~~ t>0\\
    \bm{\tilde{\omega}}(0)&=\bm{\tilde{\omega}}_0,
\end{aligned}
  \right.
  \end{equation}
 where \(\textbf{\~{F}}=\mathbf{M} \mathbf{F}\mathbf{M}^{-1}.\)
 The values on the main diagonal of \textbf{\~{F}} remain unchanged compared to those in \textbf{F}, however, let $e_i=(u_i,v_i)$, for $i\neq j$,
\[\textbf{\~F}(e_i,e_j)=\left\{\begin{aligned}
   &\frac{\sqrt{m_2(e_i)m_2(e_j)}}{m_1(u_i)},& e_j\sim u_i\sim e_i,\\
   &\frac{\sqrt{m_2(e_i)m_2(e_j)}}{m_1(v_i)},& e_j\sim v_i\sim e_i,\\
   &0,&\mbox{otherwise.}
\end{aligned}
  \right.\] 
Notice that \textbf{\~{F}} is real  symmetric, we can diagonalize  $\textbf{\~{F}} = \mathbf{P} \mathbf{D} \mathbf{P}^{T}$ where $\mathbf{D} = \operatorname{diag}(\lambda_1, \dots,\lambda_{n})$ with the eigenvalues
\[\lambda_1\leq \lambda_2\leq \cdots\leq \lambda_n,\]
and $\mathbf{P}=(\mathbf{p_1}~\mathbf{p_2}~\cdots~\mathbf{p_{n}})$ with the eigenvector $\mathbf{p_i}$ corresponding to $\lambda_i$. Sort the edges according to the magnitudes of their corresponding eigenvalues, i.e.
\[E = \{ e_1, \dots, e_n \} \text{ such that } \lambda_{e_i}\leq \lambda_{e_j},\ \forall \, i < j.\]
Let $p_{il}=\mathbf{p_i}(e_l).$
The solution to \eqref{flow1} can be rewritten as
\[
\bm{\tilde{\omega}}(t) =e^{t\textbf{\~{F}}}\bm{\tilde{\omega}}_0= \mathbf{P} \operatorname{diag}\left(e^{\lambda_1t}, \dots, e^{\lambda_{n}t}\right) \mathbf{P}^{-1} \bm{\tilde{\omega}}_0.
\]
Due to \eqref{relation}, it follows that
\[
\bm{\omega}(t) = (\mathbf{M^{-1}P})\operatorname{diag}\left(e^{\lambda_1t}, \dots, e^{\lambda_{n}t}\right) (\mathbf{M^{-1}P})^{-1} \bm{\omega}_0.
\]
Therefore, for any $e_l\in E,$
the solution decomposes as
\begin{equation}\label{solu}
    \omega(t,e_l) = \sum_{i=1}^{n}c_{i}(e_l) \exp\{\lambda_i t\},
\end{equation}
where $c_{i}(e_l):=\frac{1}{\sqrt{m_2(e_{l})}}p_{il}\sum_{j=1}^{n}p_{ij}\omega_{oj}\sqrt{m_2(e_j)}.$

By the positivity of $\omega(t,e)$ for any $e\in E$ and any $t>0$ (see Theorem \ref{th:exist}), the normalized weight can be expressed as
\begin{equation*}
   \bar{\omega}(t,e_l) = \frac{\sum_{i=1}^{n}c_{i}(e_l) \exp\{\lambda_{i} t\}}{\sum_{i=1}^{n}\bar{c_i}\exp\{\lambda_{i} t\}},
\end{equation*}  
where $\bar{c_i}:=\sum_{e\in E}c_i(e)$.

\begin{lemma}[The Perron-Frobenius Theorem]

Let \(\bm A\) be an \(n \times n\) non-negative and irreducible matrix. Then, the spectral radius \(\rho(\bm A):= \max \{ |\lambda| : \lambda \text{ is eigenvalue of } \bm A \}\) is a simple eigenvalue and \(\rho(\bm A) > 0\). Furthermore, the right eigenvector \(\mathbf{v}\) for \(\rho(\bm A)\) satisfies \(\mathbf{v} > 0\).

\end{lemma}
\begin{remark}\label{irreducible}
The irreducibility of a matrix is equivalent to: The associated digraph is strongly connected. Let \( \bm A \) be an \( n \times n \) matrix and  its associated directed graph \( G_{\bm A} \) satisfies:

\begin{itemize}
    \item The vertex set is \( \{1, 2, \ldots, n\} \),
    \item For all \( i \neq j \), there exists a directed edge \( j \to i \) if and only if  
    \(\bm A_{ij} \neq 0.\)
\end{itemize}
\end{remark}
\begin{proposition}\label{irreducibleF}
The biggest eigenvalue of  \textbf{\~{F}} is algebraically simple, and the corresponding eigenvector $\bm p_n$ is positive.
\end{proposition}
\begin{proof}
First, we claim that \textbf{\~{F}} is irreducible. Indeed, 
since \textbf{\~{F}} is symmetric, then if \( j \to i \) then \( i \to j \) on its associated directed graph \( G_{\textbf{\~{K}}} \), which can be treated as an undirected graph. From it, the associated undirected graph of \textbf{\~{F}} is exactly the line graph of $G$, which is connected due to the connectivity of $G$. By Remark \ref{irreducible}, we obtain the irreducibility of \textbf{\~{F}}.

Next, we construct a non-negative matrix from \textbf{\~{F}}. Let   \(\beta =\max_i |\bm F_{ii}|\), we construct 
\[ \textbf{\~{B}} = \textbf{\~{F}} + \beta \bm I.\]
Similar to the proof of Theorem \ref{th:exist}, we obtain that \textbf{\~{B}} is non-negative with
\[\textbf{\~{B}}_{ii}=\bm F_{ii}+\max_i |\bm F_{ii}|\geq 0 ~\mbox{and}~ \textbf{\~{B}}_{ij}=\textbf{\~{F}}_{ij}\geq0 ~\mbox{for any}~i\neq j.\]
Moreover, due to the irreducibility of \textbf{\~{F}} and $\textbf{\~{B}}_{ij}=\textbf{\~{F}}_{ij} ~\mbox{for any}~i\neq j$, we obtain \textbf{\~{B}} is irreducible. From the Perron-Frobenius Theorem, the biggest eigenvalue of \textbf{\~{B}} is positive and algebraically simple, which is denoted by $\lambda_{\max}(\textbf{\~{B}})$ and 
\[\lambda_{\max}(\textbf{\~{F}})=\lambda_{\max}(\textbf{\~{B}})-\beta.\]
Notice that \textbf{\~{B}} and \textbf{\~{F}} have the same eigenvectors. Therefore,
the right eigenvector \(\mathbf{p}_n\) corresponding to  $\lambda_{\max}(\textbf{\~{B}})$ and then to $\lambda_{\max}(\textbf{\~{F}})$  is positive.
\end{proof}

Recalling the expressions of Forman curvature \eqref{forman weight} on a general graph or Lin-Lu-Yau curvature \eqref{1-dim} on a tree, it is natural to  consider the inverse problem; that is, given the curvature vector $(\kappa_1\cdots,\kappa_n)^T$, called targeted curvature vector, to calculate the weight vector $(\omega_1\cdots,\omega_n)^T$. Similarly,  we convert the equation \eqref{forman weight}  into matrix form and then multiply it on the left by $\bm M$  to get 
\begin{equation}\label{inverse}
    (\textbf{\~{F}}+\operatorname{\bm{diag}}(\kappa_1\cdots,\kappa_n))\bm{\tilde{\omega}}=\bm 0.
\end{equation}
Based on the properties of \textbf{\~{F}}, we obtain the necessary and sufficient condition for  the linear system \eqref{inverse}
to have a positive solution.
\begin{proposition}
The  linear system \eqref{inverse} has a positive solution if and only if  the largest eigenvalue of the matrix \(\textbf{K}:=\textbf{\~{F}}+\operatorname{\bm{diag}}(\kappa_1\cdots,\kappa_n)\) is 0.
\end{proposition}
\begin{proof}
Assume \(\bm{\omega} > \mathbf{0}\) satisfies \( \textbf{K}\bm \omega = \mathbf{0} \). For any row \( i \): 
    \[
        \textbf{K}_{ii} \omega_i + \sum_{j \neq i} \textbf{K}_{ij} \omega_j = 0.
    \]
    As \( \omega_j > 0 \) and \( \textbf{K}_{ij} \geq 0 \) for \( j \neq i \), \(\sum_{j \neq i} \textbf{K}_{ij} \omega_j>0\) for the connectedness of $G$. Thus \( \textbf{K}_{ii} < 0 \) since \( \omega_i > 0 \). Let \( k =\max_i \{-\textbf{K}_{ii}\}>0\).
    Then \( k\textbf{I} +\textbf{K} \) is non-negative, symmetric, and irreducible. And,
    \[
        \lambda_{\max}(k\textbf{I} +\textbf{K}) =  k + \lambda_{\max}(\textbf{K}).
    \]
    Furthermore, \( \textbf{K}\bm \omega = \mathbf{0} \) implies \( (k\textbf{I} +\textbf{K})\bm \omega = k\bm \omega \). As \(\bm \omega> \mathbf{0}\), the Perron-Frobenius Theorem for \( k\textbf{I} +\textbf{K}\) implies that there exists no positive eigenvector corresponding to eigenvalues except $\lambda_{\max}(k\textbf{I} +\textbf{K})$, i.e., \( k = \lambda_{\max}(k\textbf{I}+\textbf{K}) \). 
    Thus, \(\lambda_{\max}(\textbf{K}) = 0 \).

On the other hand, set $l>\max_i \{-\textbf{K}_{ii}\}$.
Then \( l\textbf{I} +\textbf{K}\) is also non-negative, symmetric, and irreducible. 
By the Perron-Frobenius Theorem,  the eigenvector $\bm x$ corresponding to \( \lambda_{\max}(l\textbf{I} +\textbf{K}) \) is positive, i.e., \(\bm x > \mathbf{0}\). It follows that 
    \[
   \textbf{K}\bm x=\lambda_{\max}(\textbf{K})\bm x.
    \]
Assume \( \lambda_{\max}(\textbf{K})= 0 \). Thus, \(\bm x\) is the positive solution we desired.

\end{proof}

Furthermore, the convergence of the the normalized weight and curvature can be derived from the properties of $\textbf{\~{F}}$.
\begin{proposition}\label{the:1.1}For any $e\in E$ and a positive initial value  $w_0(e)$,
the solution $w(t,e)$ to the general Forman-Ricci flow \eqref{flow0} converges if and only if 
\[\lambda_{\max}(\textbf{\~{F}})\leq0.\]
Moreover,  the following convergence and divergence properties hold:
\begin{itemize}
     \item  If $\lambda_{\max}(\textbf{\~{F}})<0,$ then  \(\lim_{t \to \infty} w(t,e) = 0.\)
     \item 
    If $\lambda_{\max}(\textbf{\~{F}})=0$, then \(\lim_{t \to \infty} w(t,e) = c_{n}(e)>0.\) 
    \item  If $\lambda_{\max}(\textbf{\~{F}})>0,$ then  \(\lim_{t \to \infty} w(t,e) = +\infty.\)
\end{itemize}    
\end{proposition} 

\begin{proof}
For any $e\in E,$ due to Proposition \ref{irreducibleF}, we obtain  
\[c_{n}(e_l):=\frac{1}{\sqrt{m_2(e_{l})}}p_{nl}\sum_{j=1}^{n}p_{nj}\omega_{oj}\sqrt{m_2(e_j)}>0.\]
The convergence conclusions of $w(t,e)$ can be directly derived from the expression \eqref{solu} and the positivity of $c_{n}(e)$.
\end{proof}
\begin{theorem} \label{th:main}
For any $e\in E$, Forman curvature $F_w(t,e)$ on general graphs or Lin-Lu-Yau curvature $\kappa_{\omega}(t,e)$  on trees converges to  the constant $-\lambda_{\max}({\textbf{\~{F}}})$ as $t\rightarrow \infty$. And, the limit of the normalized weight exists, i.e.
    \[\lim_{t \to \infty}\bar{\omega}(t,e)=\frac{c_{n}(e)}{\sum_{e\in E}c_{n}(e)}.\]
\end{theorem}
\begin{proof}
For any $e\in E,$
\[\begin{split}
    \bar{\omega}(t,e)&= \frac{\sum_{i=1}^nc_i(e) \exp\{\lambda_{i} t\}}{\sum_{i=1}^{n}\bar{c_i}\exp\{\lambda_{(i)} t\}}\\
    &=\frac{c_{n}(e)+\sum_{i=1}^{n-1}c_i(e)e^{(\lambda_i-\lambda_{i_0})t}}{\bar{c}_{n}+\sum_{i=1}^{n-1}\bar{c_i}e^{(\lambda_i-\lambda_{i_0})t}}\\
    &\rightarrow\frac{c_{n}(e)}{\bar{c}_{n}},~~t\rightarrow \infty.
\end{split}\]
Furthermore, 
\[
\begin{split}
    F_\omega(t,e) &=-\frac{\frac{d}{dt}\omega(t,e)}{\omega(t,e)}\\
   & =-\frac{c_{n}(e)\lambda_{n}+\sum_{i=1}^{n-1}c_i(e)\lambda_ie^{(\lambda_i-\lambda_{n})t}}{c_{n}(e)+\sum_{i=1}^{n-1}c_i(e)e^{(\lambda_i-\lambda_{n})t}}\\
   &\rightarrow-\lambda_{n},~~t\rightarrow \infty.
\end{split}
\]
\end{proof}
Let $F_\omega(\infty)=\lim_{t\rightarrow\infty}F_\omega(t,e)$, we obtain the following bounds of $F_\omega(\infty)$.
\begin{corollary}For any $e\in E$, 
     \[\max_{e\in E}\left(\frac{m_2(e)}{m_1(u)}+\frac{m_2(e)}{m_1(v)}-\sum_{e'\sim x\sim e}\frac{\sqrt{m_2(e)m_2(e')}}{m_1(x)}\right)\leq 
     F_\omega(\infty)
     \leq\min_{e\in E}\left(\frac{m_2(e)}{m_1(u)}+\frac{m_2(e)}{m_1(v)}\right).\] 
\end{corollary}
\begin{proof}

For any $e,e'\in E$,
\[\frac{\omega(t,e')}{\omega(t,e)}\rightarrow \frac{c_{n}(e')}{c_{n}(e)}~~\mbox{as}~~t\rightarrow \infty.\]
It follows that, for any $e\in E$,
\begin{equation*}
  F_{\omega}(t,e)\rightarrow\frac{m_2(e)}{m_1(u)}+\frac{m_2(e)}{m_1(v)}-\sum_{e_u\sim u,e_u\neq e}\frac{m_2(e_u)}{m_1(u)}\frac{c_{n}(e_u)}{c_{n}(e)}-\sum_{e_v\sim v,e_v\neq e}\frac{m_2(e_v)}{m_1(v)}\frac{c_{n}(e_v)}{c_{n}(e)},
\end{equation*}
as $t\rightarrow \infty$.
From Theorem \ref{th:main}, we have for any $e\in E$,
\[F_\omega(\infty)=\frac{m_2(e)}{m_1(u)}+\frac{m_2(e)}{m_1(v)}-\sum_{e_u\sim u,e_u\neq e}\frac{m_2(e_u)}{m_1(u)}\frac{c_{n}(e_u)}{c_{n}(e)}-\sum_{e_v\sim v,e_v\neq e}\frac{m_2(e_v)}{m_1(v)}\frac{c_{n}(e_v)}{c_{n}(e)},\]
Due to the positivity of $c_{n}$, we get the upper bound of $F_\omega(\infty)$.

On the other hand, we apply Gerschgorin's Disk Theorem to obtain that all the eigenvalues of {\textbf{\~{F}}} lie in the union of disks
$\cup_{i=1}^nD_i,$
where
    \[ D_i = \left\{ z \in \mathbb{R} : |z-{\textbf{\~{F}}}_{ii}| \leq \sum_{j \neq i} |{\textbf{\~{F}}}_{ij}| \right\}.\]
Therefore,  we have for any $e\in E$,
\[\lambda_{\max}({\textbf{\~{F}}})\leq -\frac{m_2(e)}{m_1(u)}-\frac{m_2(e)}{m_1(v)}+\sum_{e'\sim x\sim e}\frac{\sqrt{m_2(e)m_2(e')}}{m_1(x)}.\]
From Theorem \ref{th:main}, we have the lower bounds of $F_\omega(\infty)$. 
\end{proof}

\section{Examples and simulations}
In this section, we investigate convergence of the solution to Lin-Lu-Yau Ricci flow \eqref{ollivier} on trees for special measures or on special trees.

\subsection{The case $m_1=m_2\equiv1$}
Let $T=(V,E)$ be a tree. In this case, $\mathbf{F}$ has the following simple version:
    \begin{equation}\label{examp2}
        \mathbf{F}=-2\mathbf{I}+\mathbf{B},
    \end{equation}
    where $\mathbf{B}$ is the adjacency matrix for the edge set $E$ of the tree, i.e.
 \[\textbf{B}(e_i,e_j)=\left\{\begin{aligned}
   &1,& e_j\sim e_i,\\
   &0,&\mbox{otherwise.}
\end{aligned}
  \right.\]     
Notice that $\textbf{B}$ is also the adjacency matrix for the vertex set of the line graph $L(T)$. 

This next result establishes a complete classification for the convergence of the Ricci flow on trees with uniform measures.

\begin{theorem}\label{the:4.1}
\label{coro1}
 For a tree $T=(V,E)$ with $n$ edges,  we assume that $m_1=m_2\equiv1$. Let $\omega(t,e)$ be the solution to the Ricci flow \eqref{ollivier}. Then:
 \begin{enumerate}
    \item If  $T$ is a path, then $\omega(t,e)$ for any $e\in E$  converges to $0$, and the Ricci curvature $\kappa_w(t)$ converges to a positive constant.   
    \item For $K_{1,3}$,  $w(t,e)$ for any $e\in E$  converges to a nonzero real number, and the Ricci curvature $\kappa_w(t)$ converges to $0$.
    \item If $\max_{x\in V}d(x)\geq 3$, and $T$ is not $K_{1,3}$, then $\omega(t,e)$ for any $e\in E$ diverges to infinity, and the Ricci curvature $\kappa_w(t)$ converges to a negative constant.
 \end{enumerate}
\end{theorem}

\begin{proof}
The condition  $m_2\equiv1$ ensures that  $\mathbf{F}$ is symmetric, and from \eqref{examp2}, we have 
\[\lambda_{\max}(\mathbf{F})=-2+\lambda_{\max}(\mathbf{B}).\]

For a path with $n$ edges, its line graph $P_{n}$ is also a path with \(n\) vertices.  Therefore, from spectral graph theory, we have 
\[
\lambda_{\max}(\mathbf{B}_{P_{n}}) = 2\cos\left(\frac{\pi}{n+1}\right).
\]
From Proposition \ref{the:1.1}.
\[\kappa_{\omega}(t,e)\rightarrow2\left(1-\cos\left(\frac{\pi}{n+1}\right)\right)>0, \]
and $\omega(t,e)\rightarrow 0$ as $t\rightarrow \infty$ for any $e\in E$.

For $K_{1,n}$,  we have
    \[\mathbf{F}_{K_{1,n}}=-3\mathbf{I}+\mathbf{J},\]
where $\mathbf{J}$ is the matrix with all entries equal to 1.
Therefore,
\[\lambda_{\max}(\mathbf{F}_{K_{1,n}})=n-3.\]
and its corresponding eigenvector is the vector with all components equal to $\frac{1}{\sqrt{n}}$. If $n=3,$ then $\lambda_{\max}(\mathbf{F})=0$, it follows that for any $e\in E,$ 
$\kappa_{\omega}(t,e)\rightarrow 0 $
and
$\omega(t,e)\rightarrow\frac{1}{3}\sum_{j=1}^{3}\omega_{oj}$ as $t\rightarrow \infty$.

To obtain the last conclusion, 
we first  {\it claim} that 

\[\lambda_{\max}(\mathbf{B})\geq\max_{x\in V}d(x)-1,\]
To end this, we consider a vertex $v$ of maximum degree $d:=\max_{x\in V}d(x)$ in the tree, with neighbors $u_1, u_2, \ldots, u_{d}$. In the line graph $L(T)$, all edges $e_i=\{v, u_i\}$ share the common vertex $v$. Thus, the set $\{e_1, e_2, \ldots, e_{d}\}$ induces a complete subgraph $K_d$. The largest eigenvalue of the adjacency matrix of $K_d$ is $d- 1$. Since $K_d$ is an induced subgraph of $L(T)$, its adjacency matrix $\mathbf{A}$ is a principal submatrix of $\mathbf{B}$ (the adjacency matrix of $L(T)$).  Let $S$ be the vertex set of $K_d$. Set $X=\{x\in \mathbb{R}^{n}: x=0 ~\mbox{in $S$ and $x\neq 0$ in $S^c$}~\}$. The Rayleigh quotient satisfies: for any $x\in X,$

\[\frac{x^T\mathbf{B} x}{x^T x} = \frac{y^T\mathbf{A} y}{y^T y},\]
where $y = x|_S$ (the restriction of $x$ to $S$). Therefore,
\[\lambda_{\max}(\mathbf{B}) =\max_{0\neq z \in \mathbb{R}^{n}} \frac{z^T\mathbf{B} z}{z^T z} \geq \max_{0\neq x \in X} \frac{x^T\mathbf{B} x}{x^T x} = \max_{y \neq 0} \frac{y^T\mathbf{A} y}{y^T y} = \lambda_{\max}(\mathbf{A}) = d - 1.\]
This completes the claim.
Therefore, if $\max_{x\in V}d(x)> 3$, one has 
\[\lambda_{\max}(\mathbf{F})>0. \]
Next, we consider the minimal graphs $H$ (not $K_{1,3}$) with maximum degree $3$ that contain $K_{1,3}$ as a subgraph. Its line graph is the Paw graph, i.e., $K_3$ with a pendant edge attached, and the largest eigenvalue of the adjacency matrix is larger than $2.17$.
Using a method similar to that employed in the proof of the claim, it can be shown that for any tree containing $H$ as a subgraph, the largest eigenvalue of the adjacency matrix of its line graph exceeds $2.17$. Therefore,   for a tree with maximum degree at least $3$ but not $K_{1,3}$,
$\kappa_{\omega}(t,e)\rightarrow -\lambda_{\max}(\mathbf{F})(<0) $
and
$\omega(t,e)\rightarrow\infty$ as $t\rightarrow \infty$ for any $e\in E$.

\end{proof}

\subsection{The case $\operatorname{Deg}(x)\equiv1$}
Convergence under normalized measure becomes much more difficult.  Following the earlier analysis, considering $\mathbf{\tilde{F}}$ is more convenient than considering $\mathbf{F}$. In the following examples, we require that $\operatorname{Deg}(x)\equiv1$ for any $x\in V.$
\begin{example}
    We consider a  path $p=\{v_1,v_2,\cdots,v_{n+1}\}$ with $v_i\sim v_{i+1}(i=1,\cdots,n)$. Let $e_i=(v_i,v_{i+1}), i=1,\cdots, n$. Set $m_2(e_i)=a_i$.
    By  $\operatorname{Deg}(x)=1$, the coefficient matrix $\mathbf{\tilde{F}}$ can be rewritten as 
\[\mathbf{\tilde{F}}(e_i,e_i)
=\left\{\begin{aligned}
  & -\left(\frac{a_i}{a_{i-1}+a_i}+\frac{a_i}{a_i+a_{i+1}}\right),
& i= 2,\cdots, n-1,\\
  & -1-\frac{a_1}{a_1+a_2},& i=1,\\
  &-1-\frac{a_n}{a_n+a_{n-1}},& i=n.\\
\end{aligned}
  \right.\]
  and 
\[\mathbf{\tilde{F}}(e_i,e_{i+1})=\frac{\sqrt{a_ia_{i+1}}}{a_i+a_{i+1}}.\]
Otherwise, $\mathbf{\tilde{F}}(e_i,e_{j})=0$.
For any $\bm{x}\in \mathbb{R}^n$, we obtain 
\[\bm{x}^T\mathbf{\tilde{F}}\bm{x}=-\sum_{i=1}^n\frac{1}{a_i+a_{i+1}}(\sqrt{a_i}x_i-\sqrt{a_{i+1}}x_{i+1})^2-x_1^2-x_{n}^2\leq0,\]
Furthermore,
\[\bm{x}^T\mathbf{\tilde{F}}\bm{x}=0,\]
if and only if $\bm{x}=\mathbf{0}$. It shows that the solution $\bm{\omega}(t)$ to the Ricci flow \eqref{ollivier}  on the path converges to $\bm 0$ as $t\rightarrow \infty$, and Lin-Lu-Yau curvature converges  to a positive constant.



\end{example}

\begin{example}
For a star graph $s=\{v_0,v_1,\cdots,v_n\}$, and $e_i=(v_0,v_i)$ for $i=1,\cdots,n.$ Let $m_2(e_i)=a_i$, we have $m_1(v_i)=a_i(i=1,\dots, n)$ and $m_1(v_0)=\sum_{i=1}^na_i$. Set $a=\sum_{i=1}^na_i$. Therefore, the coefficient matrix $\mathbf{\tilde{F}}$ can be rewritten as 
\[\mathbf{\tilde{F}}(e_i,e_i)=-1-\frac{a_i}{a},\]
and for any $i\neq j$, 
\[\mathbf{\tilde{F}}(e_i,e_j)=\frac{\sqrt{a_ia_j}}{a}.\]
Then we decompose $\mathbf{\tilde{F}}$ into 
\[\mathbf{\tilde{F}}=-\mathbf{I}+\frac{1}{a}(\mathbf{v}\mathbf{v}^T-2\mathbf{D}),\]
where $\mathbf{D}=\operatorname{diag}(a_1,a_2,\cdots,a_n)$ and $\bm{v}=(\sqrt{a_1},\sqrt{a_2},\cdots,\sqrt{a_n}).$
We claim that 
\begin{equation}\label{egei}
\lambda_{\max}(\bm{v}\bm{v}^T-2\mathbf{D})<a.
\end{equation}
From it, we obtain that $\mathbf{\tilde{F}}$ is negative definite, which shows that the solution $\bm{\omega}(t)$ to the Ricci flow \eqref{ollivier} on the star graph converges to $\bm{0}$ as $t\rightarrow \infty$.
In order to prove \eqref{egei}, we let 
\[f(\bm{x})=\bm{x}^T(\bm{v}\bm{v}^T-2\mathbf{D})\bm{x}-a\bm{x}^T\bm{x}\]
for any $\bm{x}\in \mathbb{R}^n$. Applying the Cauchy-Schwarz inequality to get 
\[(\mathbf{v}^T\bm{x})^2\leq \|\mathbf{v}\|^2\|\bm{x}\|^2=a\|\bm{x}\|^2,\]
and the equality holds if and only if there exists $c\in \mathbb{R}$ such that $\bm{x}=c\bm{v}$. In this case, 
\[f(\bm{x})=-2c^2\bm{v}^T\mathbf{D}\bm{v}=-2c^2a^2<0.\]
Otherwise, 
\[f(\bm{x})<-2\bm{x}^T\mathbf{D}\bm{x}\leq0.\]
Overall, we finish the claim.
Therefore, for any $e$, $\omega(t,e)\rightarrow \infty$
\[\kappa_\omega(t,e)\rightarrow -\lambda_{\max}(\mathbf{\tilde{F}})>0, \quad t\rightarrow \infty.\]
\end{example}

\subsection{Simulations}
We present several simulations of the convergence results of Ricci flow.  Figure \ref{fig:ricci_flow_example} illustrates Lin-Lu-Yau Ricci flow \eqref{ollivier} on the $K_{1,3}$ and $K_{1,6}$ under the uniform measure and the normalized measure, revealing distinct convergence behaviors. For $K_{1,3}$, under the uniform measure, the edge weights converge to a positive constant while the Ricci curvature converges to zero; in contrast, under the normalized measure, the weights decay to zero whereas the curvature converges to a positive constant. $K_{1,6}$ exhibits markedly different behavior. Under the uniform measure, the edge weights blow up to positive infinity, accompanied by convergence of the curvature to a negative constant, while under the normalized measure, the weights converge to zero and the curvature stabilizes at a positive constant. These results reveal that the convergence behavior of Ricci flow is fundamentally governed by the underlying graph structure and the choice of measure.

The second experiment simulated the normalized Lin-Lu-Yau curvature flow \eqref{flow-equation-nor} for a tree with maximum degree of $4$, see Figure \ref{fig:tree_ricci_flow 2}. Notable convergence properties are observed: (1) all edge curvatures converge to a negative limiting value, reflecting the tree's hyperbolic-like nature; (2) despite different initial weights, structurally symmetric edges converge to an identical equilibrium state,  specifically, $\omega(e_{15})$ and $\omega(e_{25})$ converge to the same value, as do $\omega(e_{67})$ and $\omega(e_{68})$, demonstrating the flow's ability to detect structural symmetries; (3) the limiting weights of different edges reflect their distinct local geometry in the tree, e.g., $e_{56}$ converges to the largest weight as it connects the most distinct branches.
\begin{figure}[H]
  \centering

  \begin{minipage}{0.48\linewidth}
    \centering
    \raisebox{0.2cm}{\textbf{3-Star} }\\
    \begin{tikzpicture}[baseline]
        \node[draw,circle,inner sep=2pt,minimum size=8mm,fill=black!10] (1) at (0,0) {1};
        \node[draw,circle,inner sep=2pt,minimum size=8mm] (2) at (90:2) {2};
        \node[draw,circle,inner sep=2pt,minimum size=8mm] (3) at (-30:2) {3};
        \node[draw,circle,inner sep=2pt,minimum size=8mm] (4) at (210:2) {4};
        \draw[thick] (1) -- (2);
        \draw[thick] (1) -- (3);
        \draw[thick] (1) -- (4);
    \end{tikzpicture}
    \vspace{6pt}
    \raisebox{0.2cm}{\textbf{(a) uniform measure} }\\
    \raisebox{0.2cm}{$m_1=m_2\equiv 1$}\\
    \includegraphics[width=\linewidth]{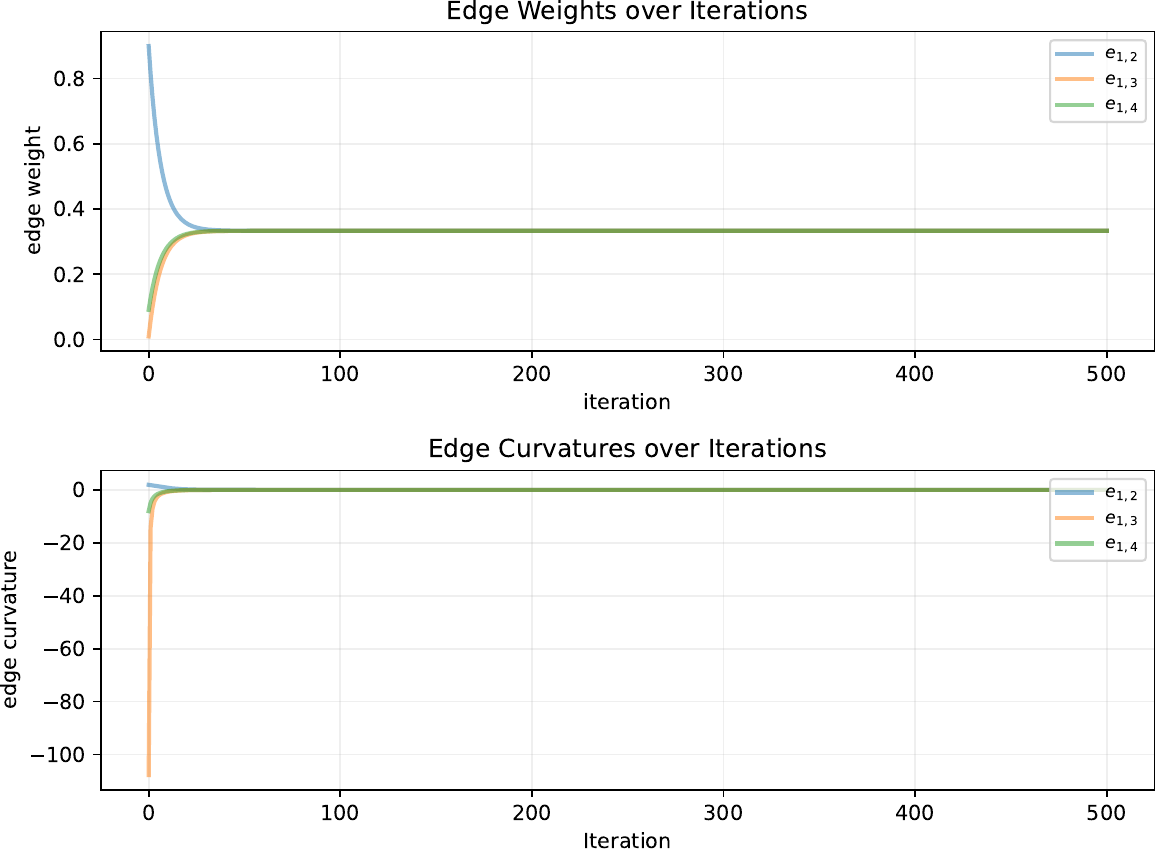}
    \vspace{6pt}
    \raisebox{0.2cm}{\textbf{(b) normalized measure}}\\
    \raisebox{0.2cm}{$m_2(e_{12})=1$, $m_2(e_{13})=2$, $m_2(e_{14})=3$}\\
    \includegraphics[width=\linewidth]{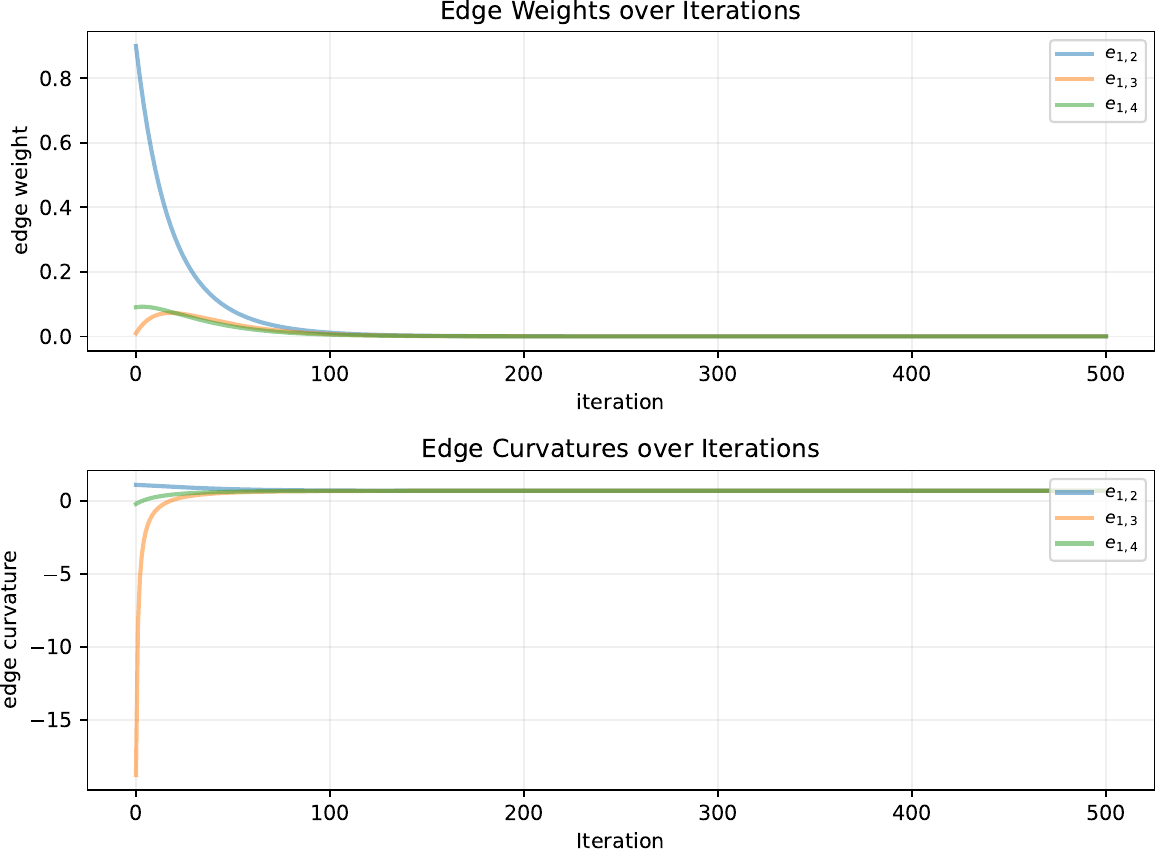}
  \end{minipage}%
  \hfill
  \begin{minipage}{0.02\linewidth}
    \centering
    \vspace{2pt}
    \begin{tikzpicture}
      \draw[gray!60, dashed, line width=0.6pt] (0cm,4cm) -- (0cm,-15cm);
    \end{tikzpicture}
  \end{minipage}%
  \hfill
  \begin{minipage}{0.48\linewidth}
    \centering
    \raisebox{0.2cm}{\textbf{6-Star} }\\
    \begin{tikzpicture}[baseline]
        \node[draw,circle,inner sep=2pt,minimum size=8mm,fill=black!10] (1) at (0,0) {1};
        \node[draw,circle,inner sep=2pt,minimum size=8mm] (2) at (180:1.732) {2};
        \node[draw,circle,inner sep=2pt,minimum size=8mm] (3) at (120:1.732) {3};
        \node[draw,circle,inner sep=2pt,minimum size=8mm] (4) at (60:1.732) {4};
        \node[draw,circle,inner sep=2pt,minimum size=8mm] (5) at (0:1.732) {5};
        \node[draw,circle,inner sep=2pt,minimum size=8mm] (6) at (-60:1.732) {6};
        \node[draw,circle,inner sep=2pt,minimum size=8mm] (7) at (-120:1.732) {7};
        \draw[thick] (1) -- (2);
        \draw[thick] (1) -- (3);
        \draw[thick] (1) -- (4);
        \draw[thick] (1) -- (5);
        \draw[thick] (1) -- (6);
        \draw[thick] (1) -- (7);
    \end{tikzpicture}
    \vspace{6pt}
    \raisebox{0.2cm}{\textbf{(c) uniform measure} }\\
    \raisebox{0.2cm}{$m_1=m_2\equiv 1$}\\
    \includegraphics[width=\linewidth]{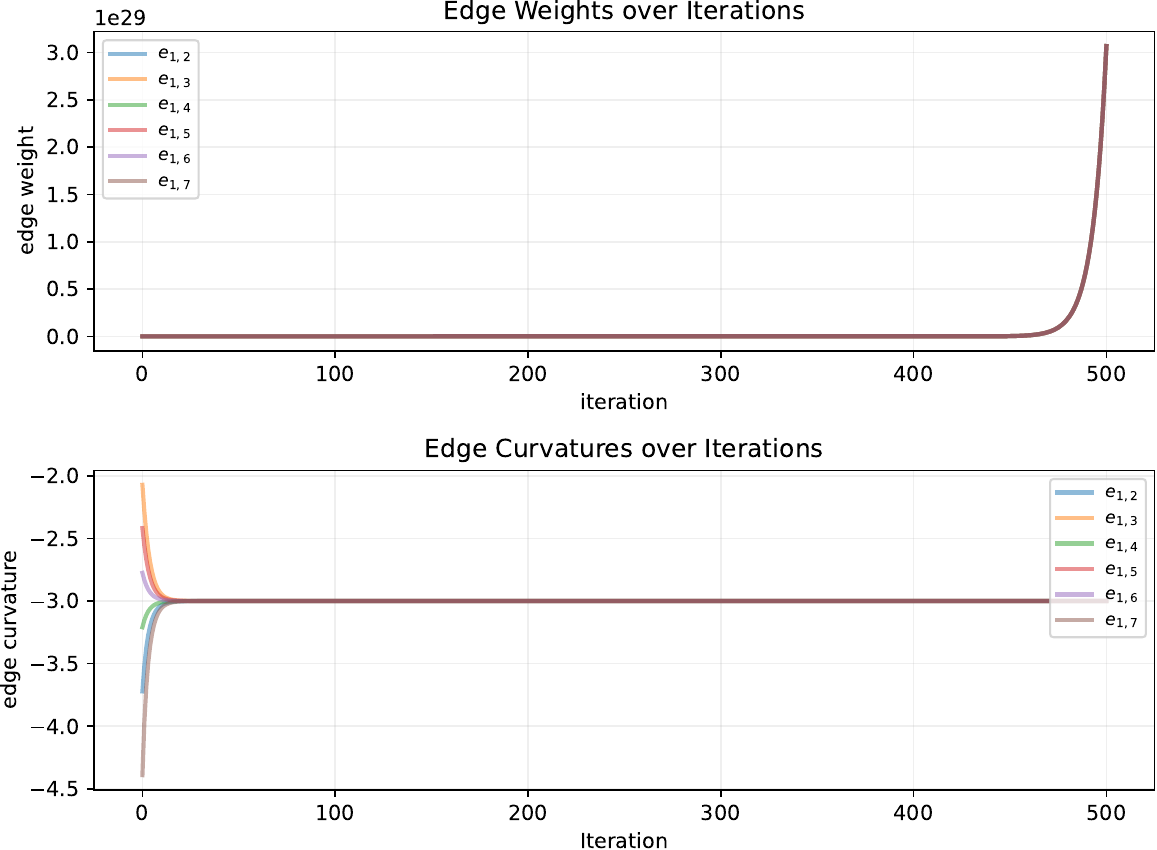}
    \vspace{6pt}
    \raisebox{0.2cm}{\textbf{(d) normalized measure}}\\
    \raisebox{0.2cm}{$m_2(e_{ij})\equiv 1, \forall e_{ij}\in E$}\\
    \includegraphics[width=\linewidth]{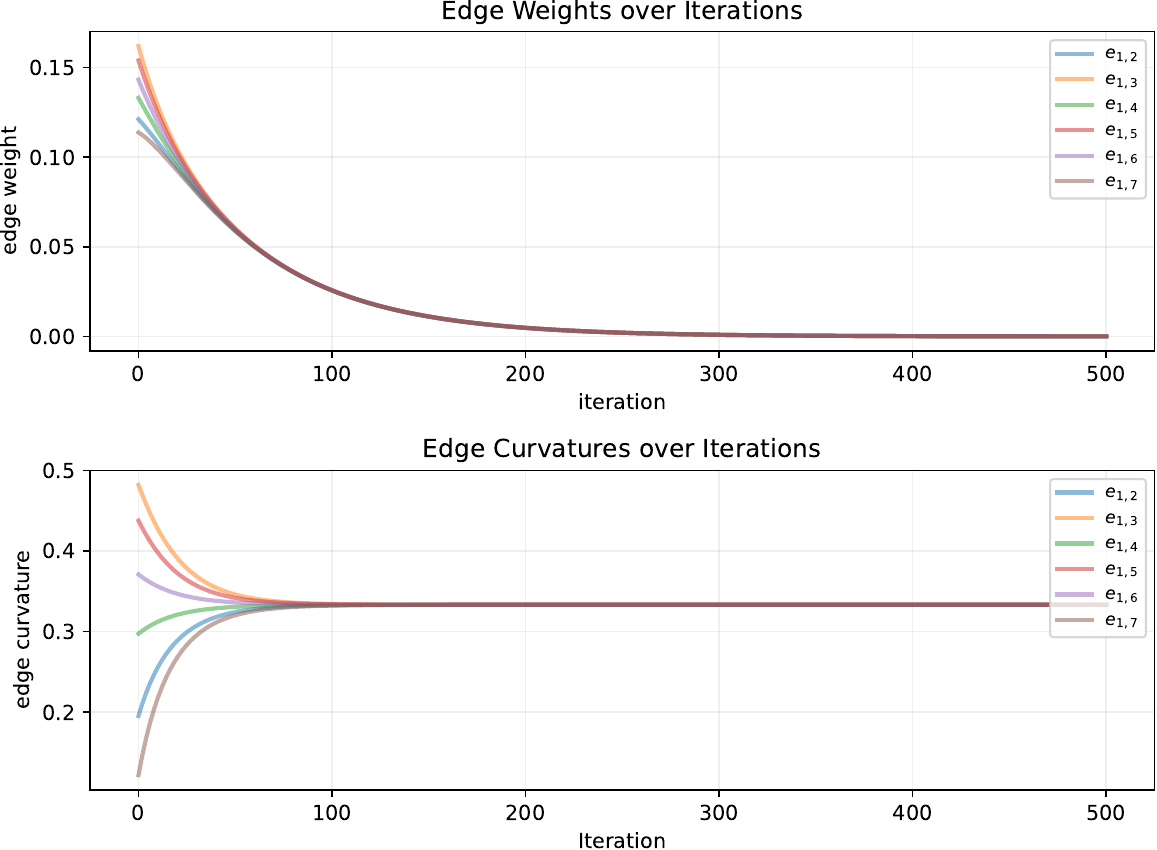}
  \end{minipage}

  \caption{
    Comparison of Ricci flow convergence under two different measures on $K_{1,3}$ and $K_{1,6}$ graphs. 
  }
  \label{fig:ricci_flow_example}
\end{figure}
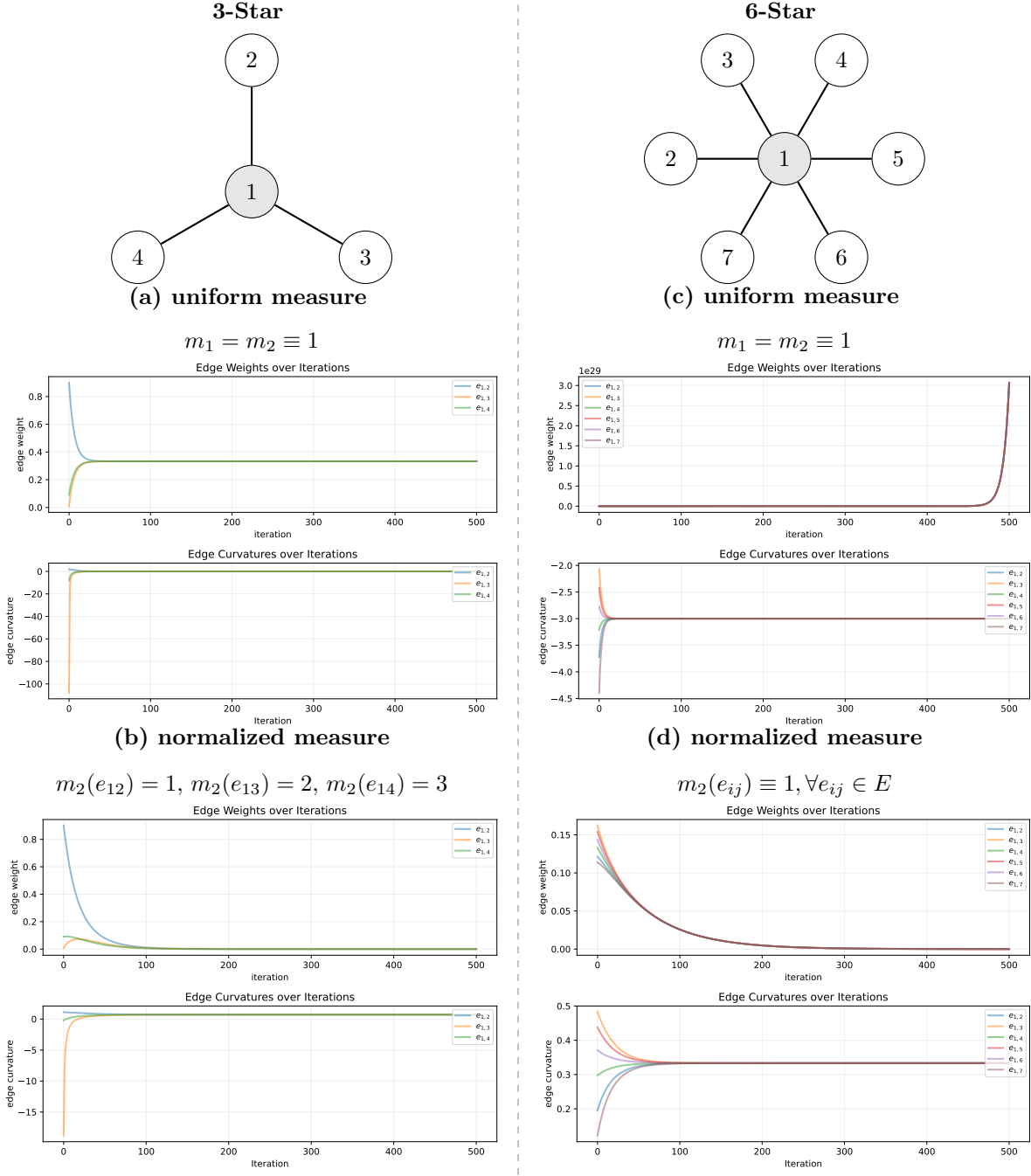

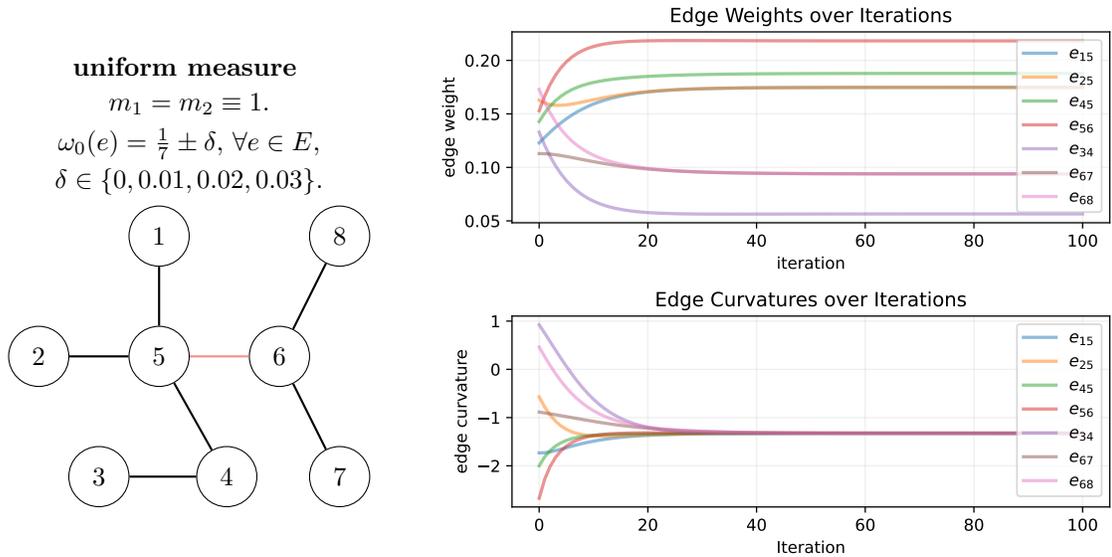
\begin{figure}[H]
\centering
\begin{minipage}{0.4\linewidth}
\centering
\raisebox{0.2cm}{\textbf{uniform measure} }\\
\raisebox{0.2cm}{$m_1=m_2\equiv 1$.}\\
\raisebox{0.2cm}{$\omega_0(e)=\frac{1}{7}\pm \delta$, $\forall e\in E$,}\\
\raisebox{0.2cm}{$\delta\in\{0, 0.01, 0.02, 0.03\}$.}\\
\begin{tikzpicture}[scale=1, every node/.style={draw,circle,inner sep=2pt,minimum size=8mm}]
\node (v5) at (0,0) {5}; 
\node (v1) at (0,1.6) {1}; 
\node (v2) at (-1.6,0) {2}; 
\node (v3) at (-0.8,-1.6) {3}; 
\node (v4) at (0.9,-1.6) {4}; 
\node (v6) at (1.6,0) {6}; 
\node (v7) at (2.4,-1.6) {7}; 
\node (v8) at (2.4,1.6) {8}; 
  \draw[thick] (v5) -- (v1);
  \draw[thick] (v5) -- (v2);
  \draw[thick] (v4) -- (v3);
  \draw[thick] (v5) -- (v4);
  \draw[thick, Salmon] (v5) -- (v6);
  \draw[thick] (v6) -- (v7);
  \draw[thick] (v6) -- (v8);
\end{tikzpicture}\\
\end{minipage}%
\hfill
\begin{minipage}{0.6\linewidth}
\centering
\includegraphics[width=\linewidth]{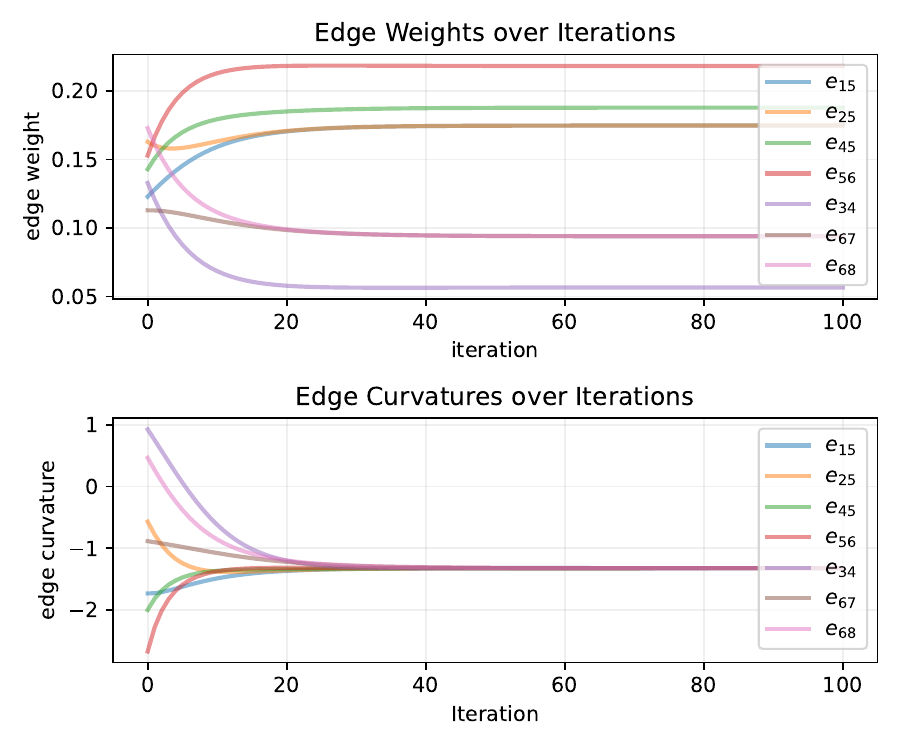}
\end{minipage}

\caption{The evolution for a tree with maximum degree of $4$. }
\label{fig:tree_ricci_flow 2}
\end{figure}

{\bf Acknowledgements:} 
S. Bai is supported by NSFC, no.12301434. S. Liu is supported by NSFC, no.12001536, 12371102.  Xin Lai is supported by the start-up research fund from Beijing Institute of Mathematical Sciences and Applications (BIMSA).

\bibliographystyle{plain}
\bibliography{citations}
\end{document}